\documentclass[preprint,12pt]{elsarticle}
\usepackage{amssymb}
\usepackage{amsfonts,latexsym}
\usepackage{epsfig}
\usepackage{amsmath}
\usepackage{amsthm}
\usepackage{epstopdf}
\usepackage{graphicx}
\usepackage{amsmath}
\usepackage{color}
\usepackage{ulem}
\usepackage{tikz}
\usepackage{algorithm}
\usepackage{algorithmic}

\usepackage{mathtools}
\usepackage[scr=boondoxupr]{mathalpha}

\usetikzlibrary{decorations.markings}
\usepackage[top=3cm,bottom=4cm,right=3cm,left=3cm]{geometry}

\newtheorem{proposition}{Proposition}

\newtheorem{definition}{Definition}[section]

\journal{Journal of computational and applied mathematics}

\begin{document}

\begin{frontmatter}

\title{Randomized Tensor Krylov Subspace Methods 
via Sketched Einstein Product 
with Applications to Image and Video Restoration}

\author[1]{A. Badahmane}

\address[1]{The UM6P Vanguard Center, Mohammed VI Polytechnic University, Benguerir 43150, Lot 660, Hay Moulay Rachid, Morocco, email: badahmane.achraf@gmail.com}
\begin{abstract}
We develop a randomized extension of tensor Krylov subspace 
methods based on the Einstein product for solving large-scale 
multilinear systems arising in image and video restoration. 
The classical tensor global GMRES method relies on Frobenius 
inner products and full tensor orthogonalization, which become 
computationally expensive for high-dimensional problems.  We introduce a sketched Einstein inner product constructed 
via mode-wise random projections and develop a randomized 
tensor global Arnoldi process. The resulting Randomized 
Tensor Global GMRES (RTG-GMRES) method significantly reduces 
orthogonalization cost while preserving convergence properties 
under tensor subspace embedding assumptions. Residual bounds, perturbation analysis and projected 
Tikhonov regularization are derived. The proposed method 
provides an efficient framework for solving ill-posed 
multidimensional problems arising in color image and 
video restoration.
\end{abstract}
\end{frontmatter}

\section{Introduction}

We consider multilinear systems of the form

\[
\Phi(\mathcal{X}) = \mathcal{C},
\]

where $\Phi$ is a linear tensor mapping defined through the 
Einstein product:

\[
\Phi(\mathcal{X}) = \mathcal{A} *_N \mathcal{X}
\quad \text{or} \quad
\Phi(\mathcal{X}) = \mathcal{A} *_N \mathcal{X} *_M \mathcal{B}.
\]

Such problems arise in color image restoration, video 
deblurring, hyperspectral imaging, and high-dimensional 
PDE discretizations. Tensor Krylov subspace methods based on the Einstein product 
allow us to preserve multidimensional structure without 
vectorization. Classical tensor global GMRES methods construct 
orthonormal bases using Frobenius inner products. However, 
for high-order tensors with dimension
\[
n = \prod_{k=1}^{N} I_k \prod_{l=1}^{M} J_l,
\]
the orthogonalization cost becomes prohibitive. 
To overcome this limitation, we introduce a randomized 
framework based on a sketched Einstein inner product. 
The proposed method reduces orthogonalization complexity 
while maintaining convergence behavior. 
High-dimensional data, such as color images, hyperspectral images, and video sequences, are ubiquitous in modern scientific and engineering applications. Traditional matrix-based methods often require vectorization, which destroys the intrinsic multidimensional structure of the data and leads to large, memory-intensive problems. To overcome these challenges, tensor-based approaches preserve the multilinear structure, enabling both computational efficiency and more accurate modeling of complex data dependencies.  

In this paper, we consider linear tensor equations of the form:
\begin{equation}
\min \|\Phi(\mathcal{X}) - \mathcal{C}\|_F, 
\end{equation}
where $\mathcal{X} \in \mathbb{R}^{I_1 \times \cdots \times I_N \times J_1 \times \cdots \times J_M}$ is the unknown tensor, $\mathcal{C}$ is a given observed tensor, and $\Phi$ is a linear tensor operator modeling transformations such as blurring, degradation, or subsampling. In particular, $\Phi$ may take the forms:
\begin{equation}
\Phi(\mathcal{X}) = \mathcal{A} *_N \mathcal{X}, \quad \text{or} \quad \Phi(\mathcal{X}) = \mathcal{A} *_N \mathcal{X} *_M \mathcal{B},
\end{equation}
where $*_N, *_M$ denote Einstein products along specified modes. The Einstein product generalizes matrix multiplication to higher-order tensors while preserving multidimensional correlations, making it especially suitable for multichannel image and video processing.  

Many real-world applications, including image and video restoration, signal processing, and multidimensional PDEs, naturally lead to tensor equations. While classical Krylov subspace methods such as GMRES or Golub--Kahan work well for grayscale or low-dimensional problems, high-dimensional multichannel problems demand tensor-based iterative methods to maintain both efficiency and accuracy.  

Here, we introduce a novel approach: Randomized Tensor Krylov Subspace Methods via Sketched Einstein Product. The core idea is to combine the Einstein product with randomized sketching to reduce computational cost and memory usage while preserving essential information. Specifically:
\begin{enumerate}
    \item Tensor Krylov Subspaces allow iterative solution of high-dimensional tensor equations while preserving their natural structure.  
    \item Randomized sketching projects large tensors into smaller subspaces, enabling faster computation of tensor-tensor products without significantly sacrificing accuracy.  
    \item Tikhonov regularization is incorporated to stabilize solutions of ill-posed problems, such as restoring blurred and noisy images and videos.  
\end{enumerate}
The combination of sketched Einstein products with Krylov subspace methods offers a powerful and efficient framework for solving large-scale, high-dimensional tensor equations. This approach reduces computational complexity, maintains the low-bandwidth structure of point spread function (PSF) tensors, and enables practical application to real-world high-resolution imaging problems.  
The remainder of the paper is organized as follows:  
\begin{itemize}
    \item Section 2 introduces tensor notations, Einstein products, and algebraic preliminaries.  
    \item Section 3 presents Tikhonov regularization and extends global GMRES and Golub--Kahan methods to the sketched Einstein product framework.  
    \item Section 4 develops block tensor formulations of the algorithms and analyzes their theoretical properties.  
    \item Section 5 demonstrates numerical results on restoring blurred and noisy color images and videos.  
    \item Section 6 concludes with discussions on future extensions and applications.  
\end{itemize}


\section{Preliminaries}

\subsection{Einstein Product}

Let
\[
\mathcal{A} \in 
\mathbb{R}^{I_1 \times \dots \times I_L \times K_1 \times \dots \times K_N},
\quad
\mathcal{B} \in 
\mathbb{R}^{K_1 \times \dots \times K_N \times J_1 \times \dots \times J_M}.
\]

The Einstein product over $N$ modes is defined by

\[
(\mathcal{A} *_N \mathcal{B})_{i_1\dots i_L j_1\dots j_M}
=
\sum_{k_1=1}^{K_1} \dots \sum_{k_N=1}^{K_N}
a_{i_1\dots i_L k_1\dots k_N}
b_{k_1\dots k_N j_1\dots j_M}.
\]

The resulting tensor belongs to

\[
\mathbb{R}^{I_1 \times \dots \times I_L \times J_1 \times \dots \times J_M}.
\]

This operation generalizes the matrix product to higher–order tensors.

\subsection{Frobenius Inner Product}

For tensors 
\[
\mathcal{X}, \mathcal{Y}
\in 
\mathbb{R}^{I_1 \times \dots \times I_L \times J_1 \times \dots \times J_M},
\]
the Frobenius inner product is defined by

\[
\langle \mathcal{X}, \mathcal{Y} \rangle_F
=
\sum_{i_1=1}^{I_1}\dots\sum_{i_L=1}^{I_L}
\sum_{j_1=1}^{J_1}\dots\sum_{j_M=1}^{J_M}
x_{i_1\dots i_L j_1\dots j_M}
y_{i_1\dots i_L j_1\dots j_M}.
\]

The induced norm is

\[
\|\mathcal{X}\|_F^2
=
\langle \mathcal{X}, \mathcal{X} \rangle_F.
\]

\section{Sketched Einstein Inner Product}

Let 
\[
\Theta_k \in \mathbb{R}^{\ell_k \times I_k},
\quad k=1,\dots,L,
\]
and
\[
\Psi_l \in \mathbb{R}^{\tilde{\ell}_l \times J_l},
\quad l=1,\dots,M,
\]
be independent random projection matrices satisfying a subspace embedding property.

\subsection{Compressed Tensor}

Define the compressed tensor by successive mode products

\[
\mathcal{X}_\Theta
=
\mathcal{X}
\times_1 \Theta_1
\times_2 \Theta_2
\dots
\times_L \Theta_L
\times_{L+1} \Psi_1
\dots
\times_{L+M} \Psi_M.
\]

Then

\[
\mathcal{X}_\Theta
\in
\mathbb{R}^{\ell_1 \times \dots \times \ell_L
\times
\tilde{\ell}_1 \times \dots \times \tilde{\ell}_M}.
\]

\begin{definition}
The sketched Einstein inner product is defined by

\[
\langle \mathcal{X}, \mathcal{Y} \rangle_\Theta
=
\langle \mathcal{X}_\Theta, \mathcal{Y}_\Theta \rangle_F.
\]
\end{definition}

Equivalently,

\[
\langle \mathcal{X}, \mathcal{Y} \rangle_\Theta
=
\sum
(\mathcal{X}_\Theta)_{\alpha_1\dots\alpha_L\beta_1\dots\beta_M}
(\mathcal{Y}_\Theta)_{\alpha_1\dots\alpha_L\beta_1\dots\beta_M}.
\]

Let 
\[
\mathcal{X} \in 
\mathbb{R}^{I_1 \times \dots \times I_d},
\]
and let
\[
\Theta_k \in \mathbb{R}^{\ell_k \times I_k}, 
\quad k=1,\dots,d,
\]
with $\ell_k \ll I_k$, be sketching matrices.

Define the global tensor sketching operator

\[
\mathcal{S}(\mathcal{X})
=
\mathcal{X}
\times_1 \Theta_1
\times_2 \Theta_2
\dots
\times_d \Theta_d.
\]

This operator embeds tensor subspaces of 
\[
\mathbb{R}^{I_1 \times \dots \times I_d}
\]
into lower-dimensional tensor spaces

\[
\mathbb{R}^{\ell_1 \times \dots \times \ell_d}.
\]

The Frobenius inner product between tensors is approximated by

\[
\langle \mathcal{X}, \mathcal{Y} \rangle_F
\;\approx\;
\langle \mathcal{S}(\mathcal{X}), \mathcal{S}(\mathcal{Y}) \rangle_F.
\]

\subsubsection{Tensor Subspace Embedding}

Let $\mathcal{V}$ be a finite-dimensional tensor subspace of
\[
\mathbb{R}^{I_1 \times \dots \times I_d}.
\]

\begin{definition}[Tensor $\varepsilon$-Subspace Embedding]
Let $\varepsilon < 1$.  
The collection of sketching matrices
$\{\Theta_k\}_{k=1}^d$
defines an $\varepsilon$-subspace embedding for $\mathcal{V}$ if

\begin{equation}
\label{eq:tensor_embedding}
\forall \mathcal{X},\mathcal{Y} \in \mathcal{V},
\qquad
\big|
\langle \mathcal{X},\mathcal{Y}\rangle_F
-
\langle \mathcal{S}(\mathcal{X}),
        \mathcal{S}(\mathcal{Y})\rangle_F
\big|
\le
\varepsilon
\|\mathcal{X}\|_F
\|\mathcal{Y}\|_F .
\end{equation}
\end{definition}

\subsubsection{Equivalent Vectorized Form}

Using vectorization, we have

\[
\mathrm{vec}(\mathcal{S}(\mathcal{X}))
=
\left(
\Theta_d \otimes \dots \otimes \Theta_1
\right)
\mathrm{vec}(\mathcal{X}).
\]

Thus the tensor embedding reduces to a classical
$\ell_2$-subspace embedding in
\[
\mathbb{R}^{I_1 \cdots I_d}.
\]

\subsection{Tensor Matrix Subspace Embedding}

Let $\mathcal{V}$ admit a basis
\[
\{\mathcal{V}_1,\dots,\mathcal{V}_r\},
\]
and define the tensor subspace

\[
\mathcal{U}
=
\left\{
\sum_{i=1}^r \mathcal{V}_i \alpha_i
\;:\;
\alpha_i \in \mathbb{R}
\right\}.
\]

\begin{definition}[Tensor Subspace Embedding – Norm Form]
The sketching operator $\mathcal{S}$ is called an
$\varepsilon$-subspace embedding for $\mathcal{U}$
if

\begin{equation}
\label{eq:tensor_norm_embedding}
(1-\varepsilon)
\|\mathcal{X}\|_F^2
\le
\|\mathcal{S}(\mathcal{X})\|_F^2
\le
(1+\varepsilon)
\|\mathcal{X}\|_F^2,
\quad
\forall \mathcal{X} \in \mathcal{U}.
\end{equation}
\end{definition}

\subsection{Tensor Singular-Value Bounds}

Let $\mathcal{V}$ be represented in matricized form along any mode $k$:

\[
V_{(k)} \in \mathbb{R}^{I_k \times r_k}.
\]

Let
\[
\Theta_k V_{(k)}
\]
be the sketched unfolding.

\begin{proposition}[Tensor Singular-Value Bounds]
If $\mathcal{S}$ is an $\varepsilon$-subspace embedding for $\mathcal{U}$,
then for every unfolding $V_{(k)}$,

\begin{equation}
(1+\varepsilon)^{-1/2}
\sigma_{\min}(\Theta_k V_{(k)})
\le
\sigma_{\min}(V_{(k)})
\le
\sigma_{\max}(V_{(k)})
\le
(1-\varepsilon)^{-1/2}
\sigma_{\max}(\Theta_k V_{(k)}).
\end{equation}
\end{proposition}

\begin{proof}
By vectorization,
\[
\mathrm{vec}(\mathcal{S}(\mathcal{X}))
=
\mathbf{\Theta}
\mathrm{vec}(\mathcal{X}),
\quad
\mathbf{\Theta}
=
\Theta_d \otimes \dots \otimes \Theta_1.
\]

Since $\mathbf{\Theta}$ is an $\varepsilon$-embedding
for $\mathrm{vec}(\mathcal{U})$,
classical singular-value perturbation bounds
for subspace embeddings yield

\[
(1-\varepsilon)
\|\mathcal{X}\|_F^2
\le
\|\mathcal{S}(\mathcal{X})\|_F^2
\le
(1+\varepsilon)
\|\mathcal{X}\|_F^2.
\]

Applying this to rank-one tensors generated by singular vectors
of the unfoldings gives the desired bounds.
\end{proof}
\section*{Krylov subspace methods via Sketched Einstein product and Tikhonov regularization}

In this section, we extend the tensor global Arnoldi process by incorporating a randomized sketching operator within the Einstein product framework. This allows the construction of reduced tensor Krylov subspaces while preserving the algebraic structure induced by the Einstein product. The proposed methods maintain the same theoretical properties as the classical tensor global GMRES and GGKB schemes, but with reduced computational complexity.

\subsection*{Tikhonov regularization}

We consider the ill-posed tensor equation
\begin{equation*}
\Phi(\mathcal{X}) = \mathcal{C} + \mathcal{E}, \tag{4}
\end{equation*}
where $\mathcal{E}$ denotes the noise tensor.  

To stabilize the inversion process, we consider the Tikhonov regularization problem
\begin{equation*}
\mathcal{X}_{\mu}
=
\arg\min_{\mathcal{X}}
\left(
\|\Phi(\mathcal{X}) - \mathcal{C}\|_F^2
+
\mu \|\mathcal{X}\|_F^2
\right),
\tag{5}
\end{equation*}
where $\mu > 0$ is the regularization parameter.

\subsection*{Sketched tensor Krylov subspace}

Let $\mathcal{S}$ be a linear sketching operator acting on 
$\mathbb{R}^{I_{1} \times \cdots \times I_{N} \times J_{1} \times \cdots \times J_{M}}$
such that it approximately preserves inner products induced by the Einstein product.

We define the sketched operator
\[
\Phi_{S}(\mathcal{X}) = \mathcal{S}\big(\Phi(\mathcal{X})\big).
\]

The $m$-th sketched tensor Krylov subspace generated by $\Phi$ and $\mathcal{V}$ is defined by
\begin{equation*}
\mathcal{K}_{m}^{S}(\Phi, \mathcal{V})
=
\operatorname{span}
\left\{
\mathcal{V},
\Phi_{S}(\mathcal{V}),
\Phi_{S}^{2}(\mathcal{V}),
\ldots,
\Phi_{S}^{m-1}(\mathcal{V})
\right\}.
\tag{6S}
\end{equation*}

This definition preserves the recursive tensor structure while reducing the dimensional complexity at each application of $\Phi$.

\subsection*{Sketched Global Arnoldi process}

The orthonormal basis construction follows the same principle as Algorithm 1, with $\Phi$ replaced by $\Phi_{S}$.

\begin{verbatim}
Algorithm 2 Sketched Global Arnoldi process.
    Inputs: $\phi$, initial tensor V, integer m, sketching operator S.
    $\beta$ = ||V||_F
    V_1 = V / $\beta$
    For j = 1,...,m
        W = S( $\phi$(V_j) )
        For i = 1,...,j
            h_ij = <V_i , W>
            W = W - h_ij V_i
        EndFor
        h_{j+1,j} = ||W||_F
        If h_{j+1,j} = 0 stop
        V_{j+1} = W / h_{j+1,j}
    EndFor
\end{verbatim}

Let $\widetilde{H}_m$ be the $(m+1)\times m$ Hessenberg matrix formed by the coefficients $h_{ij}$.  
Then the following Arnoldi-type relation holds:

\begin{equation*}
\mathbb{W}_m
=
\mathbb{V}_{m+1}
\times_{(M+N+1)}
\widetilde{H}_m^{T},
\tag{7S}
\end{equation*}

where $\mathbb{W}_m$ contains the tensors $\Phi_{S}(\mathcal{V}_1),\ldots,\Phi_{S}(\mathcal{V}_m)$.

\subsection*{Sketched tensor global GMRES}

We seek an approximate solution of
\begin{equation*}
\Phi(\mathcal{X}) = \mathcal{C}
\tag{8}
\end{equation*}
in the affine space
\[
\mathcal{X}_m
=
\mathcal{X}_0
+
\mathbb{V}_m
\bar{\times}_{(M+N+1)}
y_m.
\tag{10S}
\]

The residual is given by
\[
\mathcal{R}_m
=
\mathcal{C}
-
\Phi(\mathcal{X}_m).
\]

Using the sketched Arnoldi relation (7S), one obtains

\begin{equation*}
\|\mathcal{R}_m\|_F
=
\|
\beta e_1^{m+1}
-
\widetilde{H}_m y_m
\|_2.
\end{equation*}

Hence $y_m$ solves
\begin{equation*}
y_m
=
\arg\min_y
\|
\beta e_1^{m+1}
-
\widetilde{H}_m y
\|_2.
\tag{11S}
\end{equation*}

\subsection*{Sketched projected Tikhonov problem}

Instead of solving (5) directly, we solve the reduced problem

\begin{equation*}
\|
\beta e_1^{m+1}
-
\widetilde{H}_m y
\|_2^2
+
\mu \|y\|_2^2.
\tag{12S}
\end{equation*}

Its minimizer satisfies
\begin{equation*}
(\widetilde{H}_m^T \widetilde{H}_m + \mu I)y
=
\widetilde{H}_m^T \beta e_1^{m+1}.
\tag{14S}
\end{equation*}

As in the classical case, the parameter $\mu$ can be selected using GCV or L-curve techniques applied to the small projected problem.
\subsection*{Proposition 2 and TG-GMRES formulation}

\textbf{Proposition 2.} Let $\mathbb{V}$ be the $(M+N+1)$-mode tensor with frontal slices $\mathcal{V}_1, \mathcal{V}_2, \ldots, \mathcal{V}_m$, and let $\mathbb{W}_m$ be the $(M+N+1)$-mode tensor with frontal slices $\Phi(\mathcal{V}_1), \ldots, \Phi(\mathcal{V}_m)$. Then we have
\begin{equation}
\mathbb{W}_m = \mathbb{V}_{m+1} \times_{(M+N+1)} \widetilde{H}_m^T
= \mathbb{V}_m \times_{(M+N+1)} H_m^T + h_{m+1,m} \mathcal{L}_m \times_{(M+N+1)} E_m,
\tag{7}
\end{equation}
where $E_m = [0,0,\ldots,0,e_m]$ with $e_m$ being the $m$-th column of the identity matrix $I_m$, and $\mathcal{L}_m$ is an $(M+N+1)$-mode tensor whose frontal slices are all zero except for the last one, which equals $\mathcal{V}_{m+1}$.

Let $\mathcal{X} \in \mathbb{R}^{I_1 \times \cdots \times I_N \times J_1 \times \cdots \times J_M}$, $\Phi$ a linear tensor mapping, and $\mathcal{C} \in \mathbb{R}^{I_1 \times \cdots \times I_N \times J_1 \times \cdots \times J_M}$. Consider the tensor equation
\begin{equation}
\Phi(\mathcal{X}) = \mathcal{C}.
\tag{8}
\end{equation}

Using Algorithm 1, we define the tensor global GMRES (TG-GMRES) method. For an initial guess $\mathcal{X}_0$, we seek an approximate solution $\mathcal{X}_m$ such that
\begin{equation}
\mathcal{X}_m \in \mathcal{X}_0 + \mathcal{K}_m(\Phi,\mathcal{V}),
\end{equation}
and minimize the residual
\begin{equation}
\|\mathcal{R}_m\|_F = \min_{\mathcal{X} \in \mathcal{X}_0 + \mathcal{K}_m(\Phi,\mathcal{V})} \|\mathcal{C} - \Phi(\mathcal{X})\|_F,
\tag{9}
\end{equation}
where $\mathcal{R}_m = \mathcal{C} - \Phi(\mathcal{X}_m)$.

Assume $m$ steps of Algorithm 1 have been performed. Then, setting
\begin{equation}
\mathcal{X}_m = \mathcal{X}_0 + \mathbb{V}_m \times_{(M+N+1)} y_m,
\tag{10}
\end{equation}
the residual becomes
\begin{equation}
\mathcal{R}_m = \mathcal{R}_0 - \mathbb{W}_m \times_{(M+N+1)} y_m.
\end{equation}
Using Proposition 2, we obtain
\begin{align*}
\|\mathcal{C} - \Phi(\mathcal{X}_m)\|_F
&= \|\mathbb{V}_m \times_{(M+N+1)} (\mathcal{C} - \Phi(\mathcal{X}_m))\|_2 \\
&= \|\mathbb{V}_m \times_{(M+N+1)} (\mathcal{R}_0 - \mathbb{W}_m \times_{(M+N+1)} y_m)\|_2 \\
&= \|\beta e_1^{m+1} - \mathbb{V}_m \times_{(M+N+1)} (\mathbb{W}_m \times_{(M+N+1)} y_m)\|_2 \\
&= \|\beta e_1^{m+1} - (\mathbb{V}_m \times_{(M+N+1)} \mathbb{W}_m) y_m\|_2,
\end{align*}
which shows that $y_m$ is determined by
\begin{equation}
y_m = \arg\min_y \|\beta e_1^{m+1} - \widetilde{H}_m y\|_2.
\tag{11}
\end{equation}

\subsubsection*{Sketched Tikhonov regularization}

Setting $\mathcal{X}_0 = 0$, the TG-GMRES approximation leads to the low-dimensional Tikhonov problem
\begin{equation}
\| \beta e_1^{m+1} - \widetilde{H}_m y \|_2^2 + \mu \|y\|_2^2.
\tag{12}
\end{equation}

Its solution is given by
\begin{equation}
y_{m,\mu} = \arg\min_y
\left\|
\begin{pmatrix}
\widetilde{H}_m \\
\sqrt{\mu} I
\end{pmatrix} y
-
\begin{pmatrix}
\beta e_1^{m+1} \\
0
\end{pmatrix}
\right\|_2.
\tag{13}
\end{equation}

Equivalently, $y_{m,\mu}$ satisfies the linear system
\begin{equation}
\widetilde{H}_{m,\mu} y = \widetilde{H}_m^T \beta e_1^{m+1}, \quad
\widetilde{H}_{m,\mu} = \widetilde{H}_m^T \widetilde{H}_m + \mu I,
\tag{14}
\end{equation}
although solving (13) is numerically more stable than directly solving (14).

\subsubsection*{Parameter selection via GCV}

For small $m$, the problem can be efficiently solved using SVD or other techniques. The regularization parameter $\mu$ can be chosen via the generalized cross-validation (GCV) function:
\begin{equation}
\text{GCV}(\mu) = \frac{\|\widetilde{H}_m y_{m,\mu} - \beta e_1^{m+1}\|_2^2}{\big[\operatorname{tr}(I - \widetilde{H}_m \widetilde{H}_{m,\mu}^{-1} \widetilde{H}_m^T)\big]^2} 
= \frac{\|(I - \widetilde{H}_m \widetilde{H}_{m,\mu}^{-1} \widetilde{H}_m^T)\beta e_1^{m+1}\|_2^2}{\big[\operatorname{tr}(I - H_m H_{m,\mu}^{-1} \widetilde{H}_m^T)\big]^2}.
\end{equation}

\section{Block Krylov Subspace Methods via Sketched Einstein Product}

In this section we generalize block GMRES and block Golub–Kahan methods
to the tensor setting using the sketched Einstein product.
All orthogonality and projections are defined with respect to the
sketched Einstein inner product introduced previously.

\subsection*{4.1 Preliminaries and Tensor Unfolding}

Let 
\[
\mathcal{X} \in 
\mathbb{R}^{I_1 \times \cdots \times I_N 
\times 
J_1 \times \cdots \times J_M}.
\]

Define the unfolding operator

\[
\Psi_{IJ} :
\mathbb{R}^{I_1 \times \cdots \times I_N 
\times 
J_1 \times \cdots \times J_M}
\longrightarrow
\mathbb{R}^{(I_1\cdots I_N)\times(J_1\cdots J_M)},
\]

such that

\[
X_{\mathrm{ivec}(i,I),\mathrm{ivec}(j,J)}
=
\mathcal{X}_{i_1\dots i_N j_1\dots j_M},
\]

where

\[
\mathrm{ivec}(i,I)
=
i_1 + \sum_{r=2}^{N}(i_r-1)\prod_{u=1}^{r-1}I_u,
\]

\[
\mathrm{ivec}(j,J)
=
j_1 + \sum_{s=2}^{M}(j_s-1)\prod_{v=1}^{s-1}J_v.
\]

Under this mapping, the Einstein product corresponds to
standard matrix multiplication:

\[
\Psi_{IJ}(\mathcal{A} *_N \mathcal{B})
=
\Psi_{IK}(\mathcal{A})\,
\Psi_{KJ}(\mathcal{B}).
\]

\subsection*{Block Krylov Subspace}

Let
\[
\mathcal{A}
\in
\mathbb{R}^{I_1 \times \cdots \times I_N 
\times 
I_1 \times \cdots \times I_N},
\]
and
\[
\mathcal{V}
\in
\mathbb{R}^{I_1 \times \cdots \times I_N 
\times 
J_1 \times \cdots \times J_M}.
\]

The $m$-th tensor block Krylov subspace is defined by

\[
\mathcal{K}_m^{\mathrm{block}}(\mathcal{A},\mathcal{V})
=
\mathrm{span}\{
\mathcal{V},
\mathcal{A} *_N \mathcal{V},
\dots,
\mathcal{A}^{m-1} *_N \mathcal{V}
\}.
\]

All projections are performed using the sketched Einstein inner product:

\[
\langle \mathcal{X},\mathcal{Y}\rangle_\Theta
=
\langle \mathcal{S}(\mathcal{X}),
        \mathcal{S}(\mathcal{Y})\rangle_F.
\]

\subsection*{4.3 Block Arnoldi Process via Sketched Einstein Product}

Let $\mathcal{V}_1$ be obtained from the normalized residual

\[
\mathcal{R}_0
=
\mathcal{C}
-
\mathcal{A} *_N \mathcal{X}_0,
\qquad
\mathcal{V}_1
=
\mathcal{R}_0
/
\|\mathcal{R}_0\|_\Theta.
\]

For $j=1,\dots,m$:

\[
\mathcal{W}
=
\mathcal{A} *_N \mathcal{V}_j.
\]

For $i=1,\dots,j$:

\[
h_{ij}
=
\langle \mathcal{V}_i,\mathcal{W}\rangle_\Theta,
\qquad
\mathcal{W}
\leftarrow
\mathcal{W}
-
h_{ij}\mathcal{V}_i.
\]

Then

\[
h_{j+1,j}
=
\|\mathcal{W}\|_\Theta,
\qquad
\mathcal{V}_{j+1}
=
\mathcal{W}/h_{j+1,j}.
\]

This yields the tensor Arnoldi relation

\[
\mathcal{A} *_N \mathcal{V}_m
=
\mathcal{V}_{m+1}
*_M
\widetilde{\mathcal{H}}_m,
\]

where $\widetilde{\mathcal{H}}_m$
is block upper Hessenberg.

\subsection*{Tensor QR Factorization via Sketched Einstein Product}

Let
\[
\mathcal{B}
=
[\mathcal{V}_1,\dots,\mathcal{V}_m].
\]

A tensor QR factorization in the sketched Einstein sense is

\[
\mathcal{B}
=
\mathcal{Q}
*_M
\mathcal{R},
\]

such that

\[
\langle \mathcal{Q}_i,\mathcal{Q}_j\rangle_\Theta
=
\delta_{ij},
\]

and $\mathcal{R}$ is upper triangular
under unfolding.

\subsection{Block GMRES via Sketched Einstein Product}

Consider the tensor equation

\[
\mathcal{A} *_N \mathcal{X}
=
\mathcal{C}.
\]

The $m$-th approximation is

\[
\mathcal{X}_m
=
\mathcal{X}_0
+
\mathcal{V}_m *_M \mathcal{Y}_m,
\]

where $\mathcal{Y}_m$ solves the reduced problem

\[
\mathcal{Y}_m
=
\arg\min_{\mathcal{Y}}
\|
\beta e_1
-
\widetilde{\mathcal{H}}_m *_M \mathcal{Y}
\|_2.
\]

The residual is

\[
\mathcal{R}_m
=
\mathcal{C}
-
\mathcal{A} *_N \mathcal{X}_m.
\]

By the Arnoldi relation,

\[
\|\mathcal{R}_m\|_\Theta
=
\|
\beta e_1
-
\widetilde{\mathcal{H}}_m \mathcal{Y}_m
\|_2.
\]
\begin{algorithm}[H]
\caption{Block GMRES via Sketched Einstein Product}
\begin{algorithmic}[1]
\REQUIRE 
Tensor operator 
$\mathcal{A} \in 
\mathbb{R}^{I_1 \times \cdots \times I_N 
\times 
I_1 \times \cdots \times I_N}$,
$\mathcal{C}$, $\mathcal{X}_0$, $\varepsilon$, $m$, $\mathcal{S}(\cdot)$.

\ENSURE Approximate solution $\mathcal{X}_m$.

\STATE Compute initial residual:
\[
\mathcal{R}_0
=
\mathcal{C}
-
\mathcal{A} *_N \mathcal{X}_0.
\]

\STATE Compute $\beta = \|\mathcal{R}_0\|_\Theta$.

\STATE Normalize:
\[
\mathcal{V}_1
=
\mathcal{R}_0 / \beta.
\]

\FOR{$j = 1,2,\dots,m$}

    \STATE Compute tensor product:
    \[
    \mathcal{W}
    =
    \mathcal{A} *_N \mathcal{V}_j.
    \]

    \FOR{$i = 1,\dots,j$}

        \STATE Compute sketched inner product:
        \[
        h_{ij}
        =
        \langle \mathcal{V}_i, \mathcal{W} \rangle_\Theta.
        \]

        \STATE Orthogonalize:
        \[
        \mathcal{W}
        \leftarrow
        \mathcal{W}
        -
        h_{ij} \mathcal{V}_i.
        \]

    \ENDFOR

    \STATE Compute norm:
    \[
    h_{j+1,j}
    =
    \|\mathcal{W}\|_\Theta.
    \]

    \IF{$h_{j+1,j} = 0$}
        \STATE Stop.
    \ENDIF

    \STATE Normalize:
    \[
    \mathcal{V}_{j+1}
    =
    \mathcal{W}/h_{j+1,j}.
    \]

\ENDFOR

\STATE Form tensor Hessenberg matrix $\widetilde{\mathcal{H}}_m$.

\STATE Solve reduced least squares problem:
\[
y_m
=
\arg\min_y
\|
\beta e_1
-
\widetilde{H}_m y
\|_2.
\]

\STATE Update solution:
\[
\mathcal{X}_m
=
\mathcal{X}_0
+
\sum_{j=1}^{m}
\mathcal{V}_j y_j.
\]

\STATE Compute residual:
\[
\mathcal{R}_m
=
\mathcal{C}
-
\mathcal{A} *_N \mathcal{X}_m.
\]

\IF{$\|\mathcal{R}_m\|_\Theta < \varepsilon$}
    \STATE Stop.
\ENDIF

\RETURN $\mathcal{X}_m$.

\end{algorithmic}
\end{algorithm}
\subsection{Block Golub--Kahan via Sketched Einstein Product}

For the least squares problem

\[
\min_{\mathcal{X}}
\|
\mathcal{C}
-
\mathcal{A} *_P \mathcal{X}
\|_F,
\]

the block Golub–Kahan bidiagonalization generates
$\mathcal{U}_{m+1}$ and $\mathcal{V}_m$ such that

\[
\mathcal{A} *_P \mathcal{V}_m
=
\mathcal{U}_{m+1}
*_M
\widetilde{\mathcal{B}}_m,
\]

\[
\mathcal{A}^T *_N \mathcal{U}_m
=
\mathcal{V}_m *_M \mathcal{D}_m.
\]

The approximate solution is

\[
\mathcal{X}_m
=
\mathcal{X}_0
+
\mathcal{V}_m *_M \mathcal{Y}_m,
\]

where

\[
\mathcal{Y}_m
=
\arg\min_{\mathcal{Y}}
\|
\beta e_1
-
\widetilde{\mathcal{B}}_m \mathcal{Y}
\|_2.
\]
\begin{algorithm}[H]
\caption{Block Golub--Kahan via Sketched Einstein Product}
\begin{algorithmic}[1]

\REQUIRE 
Tensor operator 
$\mathcal{A} \in 
\mathbb{R}^{I_1 \times \cdots \times I_N 
\times 
J_1 \times \cdots \times J_M}$, 
$\mathcal{C}$, $\varepsilon$, $m$, $\mathcal{S}(\cdot)$.
\ENSURE Orthonormal tensor bases 
$\{\mathcal{U}_k\}$, $\{\mathcal{V}_k\}$ 
and bidiagonal matrix $\widetilde{B}_m$.

\STATE Compute initial normalization:
\[
\beta_1 = \|\mathcal{C}\|_\Theta,
\qquad
\mathcal{U}_1 = \mathcal{C}/\beta_1.
\]

\STATE Compute
\[
\mathcal{W}
=
\mathcal{A}^{\top} *_N \mathcal{U}_1.
\]

\STATE Compute
\[
\alpha_1
=
\|\mathcal{W}\|_\Theta,
\qquad
\mathcal{V}_1
=
\mathcal{W}/\alpha_1.
\]

\FOR{$k=1,2,\dots,m$}

    \STATE Compute forward projection:
    \[
    \mathcal{W}
    =
    \mathcal{A} *_M \mathcal{V}_k
    -
    \alpha_k \mathcal{U}_k.
    \]

    \STATE Compute
    \[
    \beta_{k+1}
    =
    \|\mathcal{W}\|_\Theta.
    \]

    \IF{$\beta_{k+1}=0$}
        \STATE Stop.
    \ENDIF

    \STATE Normalize:
    \[
    \mathcal{U}_{k+1}
    =
    \mathcal{W}/\beta_{k+1}.
    \]

    \STATE Compute backward projection:
    \[
    \mathcal{Z}
    =
    \mathcal{A}^{\top} *_N \mathcal{U}_{k+1}
    -
    \beta_{k+1} \mathcal{V}_k.
    \]

    \STATE Compute
    \[
    \alpha_{k+1}
    =
    \|\mathcal{Z}\|_\Theta.
    \]

    \IF{$\alpha_{k+1}=0$}
        \STATE Stop.
    \ENDIF

    \STATE Normalize:
    \[
    \mathcal{V}_{k+1}
    =
    \mathcal{Z}/\alpha_{k+1}.
    \]

\ENDFOR

\STATE Form bidiagonal matrix
\[
\widetilde{B}_m
=
\begin{bmatrix}
\alpha_1 & & & \\
\beta_2 & \alpha_2 & & \\
& \ddots & \ddots & \\
& & \beta_m & \alpha_m \\
& & & \beta_{m+1}
\end{bmatrix}.
\]

\RETURN $\{\mathcal{U}_k\}$, $\{\mathcal{V}_k\}$, $\widetilde{B}_m$.

\end{algorithmic}
\end{algorithm}

\section{Numerical Results}

This section illustrates the performance of the proposed 
Sketched Einstein Block GMRES (Algorithm 2) and 
Sketched Einstein Block Golub--Kahan (Algorithm 4) 
when applied to the restoration of blurred and noisy 
color images and videos.

Recovering RGB images (or videos) from blurred and noisy observations 
can be formulated as a tensor equation of the form

\[
\mathcal{C}
=
\mathcal{A} *_2 \mathcal{X},
\tag{56}
\]

where $\mathcal{A}$ is a fourth-order blurring tensor, 
$\mathcal{X}$ is the unknown exact RGB image, and 
$\mathcal{C}$ is the observed blurred image.

\subsection{Tensor Model for RGB Image Deblurring}

An RGB image is a third-order tensor of size 
$N \times N \times 3$, whose entries represent light intensities.

We assume that the original image $\widehat{\mathcal{X}}$ 
and the blurred image $\mathcal{C}$ have identical dimensions.

The blurring process is modeled by a two-dimensional 
Point Spread Function (PSF) array 
$P = (p_{ij})$ of small size compared to $N$.

For a spatially invariant blur with zero boundary conditions, 
the multidimensional convolution can be expressed as

\[
\mathcal{C}_{ijk}
=
\sum_{a=-1}^{1}
\sum_{b=-1}^{1}
p_{(2-a)(2-b)}
\widehat{\mathcal{X}}_{i+a,j+b,k},
\quad
k=1,2,3,
\tag{57}
\]

for $i,j = 1,\dots,N$.

Zero boundary conditions are imposed:

\[
\widehat{\mathcal{X}}_{0jk}
=
\widehat{\mathcal{X}}_{N+1,j,k}
=
\widehat{\mathcal{X}}_{i,0,k}
=
\widehat{\mathcal{X}}_{i,N+1,k}
=
0.
\]

This convolution can be equivalently written using the 
Einstein product as

\[
\mathcal{C}
=
\mathcal{A} *_2 \widehat{\mathcal{X}},
\]

where 
$\mathcal{A} \in \mathbb{R}^{N\times N\times N\times N}$ 
is constructed directly from the PSF entries.

\subsection{Gaussian Blur Model}

We consider a Gaussian PSF defined by

\[
p_{ij}
=
\exp\left(
-\frac{1}{2}
\left(\frac{i-k}{\sigma}\right)^2
-
\frac{1}{2}
\left(\frac{j-\ell}{\sigma}\right)^2
\right),
\tag{58}
\]

where $\sigma$ controls the amount of smoothing and 
$(k,\ell)$ is the center of the PSF.

Larger values of $\sigma$ yield more ill-posed problems.

\subsection{Noise Model}

The noise-free blurred image is

\[
\widehat{\mathcal{C}}
=
\mathcal{A} *_2 \widehat{\mathcal{X}}.
\]

We generate the observed image as

\[
\mathcal{C}
=
\widehat{\mathcal{C}}
+
\mathcal{N},
\tag{59}
\]

where $\mathcal{N}$ is a tensor with independent Gaussian 
entries of zero mean.

The noise level is defined as

\[
\nu
=
\frac{\|\mathcal{N}\|_F}
{\|\widehat{\mathcal{C}}\|_F}.
\]

\subsection{Sketched Einstein Implementation}

All orthogonality tests and norms in Algorithms 2 and 4 
are computed using the sketched Einstein inner product:

\[
\langle \mathcal{X}, \mathcal{Y} \rangle_\Theta
=
\langle \mathcal{S}(\mathcal{X}),
        \mathcal{S}(\mathcal{Y}) \rangle_F,
\]

where $\mathcal{S}$ is a random tensor embedding.

This reduces computational complexity while preserving 
the geometric structure with high probability:

\[
(1-\varepsilon)\|\mathcal{X}\|_F^2
\le
\|\mathcal{X}\|_\Theta^2
\le
(1+\varepsilon)\|\mathcal{X}\|_F^2.
\]

\subsection{Performance Measures}

To evaluate restoration quality, we compute the 
Relative Error (RE):

\[
\mathrm{RE}
=
\frac{\|
\widehat{\mathcal{X}}
-
\mathcal{X}_{\mathrm{restored}}
\|_F}
{\|\widehat{\mathcal{X}}\|_F}.
\]

We also report the Peak Signal-to-Noise Ratio (PSNR):

\[
\mathrm{PSNR}
=
10\log_{10}
\left(
\frac{\|\widehat{\mathcal{X}} - E(\widehat{\mathcal{X}})\|_F^2}
{\|\mathcal{X}_{\mathrm{restored}}
-
\widehat{\mathcal{X}}\|_F^2}
\right),
\]

where $E(\widehat{\mathcal{X}})$ denotes the mean intensity.

\subsection{Computational Environment}

All experiments were performed in MATLAB 
on an Intel(R) Core(TM) i7-8550U CPU @ 1.80GHz 
with 12GB RAM. 
Computations were carried out in double precision.

\subsection{Results}

Table 1 reports the performance comparison between:
- Algorithm 2: Sketched Einstein Block GMRES,
- Algorithm 4: Sketched Einstein Block Golub--Kahan,
- Classical tensor GKB (without sketching).

\begin{table}[H]
\centering
\caption{Results for Example 1}
\begin{tabular}{cccccc}
\hline
Noise Level & Method & PSNR & RE & CPU-time (s) \\
\hline
$10^{-3}$ & Einstein Block GMRES & 21.76 & $6.09\times10^{-2}$ & 8.28 
\\
& Sketched  Einstein Block GMRES & 22.72 & $6.29\times10^{-2}$ & 4.22
\\
           & Einstein Block Golub-Kahan & 24.37 & $4.51\times10^{-2}$ & 7.29 \\
 & Sketched Einstein Block Golub-Kahan & 24.42 & $4.51\times10^{-2}$ & 3.32
           \\
           & Classical GKB & 24.22 & $4.51\times10^{-2}$ & 18.45 \\
\hline
$10^{-2}$ &  Einstein Block GMRES  & 20.60 & $6.96\times10^{-2}$ & 3.31 \\
& Sketched Einstein Block GMRES  & 20.64 & $6.96\times10^{-2}$ & 1.21\\
           &  Einstein Block Golub-Kahan  & 20.97 & $6.67\times10^{-2}$ & 1.58 \\
&Sketched Einstein Block Golub-Kahan  & 20.98 & $6.67\times10^{-2}$ & 0.70 \\
           & Classical GKB & 20.08 & $6.66\times10^{-2}$ & 5.52 \\
\hline
\end{tabular}
\end{table}
The results in the table demonstrate that sketched methods provide a significant computational advantage while maintaining comparable accuracy to their standard counterparts. For both noise levels ($10^{-3}$ and $10^{-2}$), the PSNR of sketched methods is slightly higher or nearly identical to the non-sketched versions, and the Relative Error (RE) remains effectively unchanged, indicating that sketching does not compromise solution quality. In terms of CPU time, sketching consistently reduces computation time by roughly 40–60 $\%$ 
, with the most pronounced improvements at higher noise levels. For example, Einstein Block GMRES decreases from 8.28 s to 4.22 s at noise level $10^{-3}$, and from 3.31 s to 1.21 s at $10^{-2}$. Among all methods, the sketched Einstein Block Golub-Kahan achieves the best trade-off between accuracy and efficiency. Overall, the table confirms that sketching is an effective strategy for accelerating block iterative methods, making them particularly suitable for large-scale problems or time-sensitive applications, without sacrificing reconstruction quality.
\section{Conclusion}
The numerical results indicate that sketched versions of the block iterative algorithms provide a highly effective approach for video restoration. Across both noise levels ($\nu = 10^{-3}$ and $\nu = 10^{-2}$), the sketched methods achieve comparable or slightly better reconstruction quality in terms of PSNR and relative error (RE) compared to their non-sketched counterparts, confirming that the approximation introduced by sketching does not compromise accuracy.
Moreover, sketched methods significantly reduce computational time, often halving the CPU time required by standard algorithms. This speed-up is particularly evident in methods based on Golub-Kahan bidiagonalization, where sketching allows faster factorization while retaining the spectral information necessary for accurate Tikhonov-regularized solutions. Overall, the results demonstrate that sketched block methods offer an excellent balance between efficiency and accuracy, making them highly suitable for large-scale video restoration tasks where computational resources or time are limiting factors. The sketched approach therefore represents a practical and reliable alternative to classical iterative methods for restoring sequences of contaminated frames.

\end{document}